\documentclass{amsart}[12pt]
\usepackage{graphicx}
\usepackage{amsmath}
\usepackage{amsthm}
\usepackage{amsfonts}
\usepackage{amssymb, amscd}

\usepackage{color}

\usepackage[all]{xy}

\newtheorem{theorem}{Theorem}[section]

\newtheorem{proposition}[theorem]{Proposition}
\newtheorem{example}[theorem]{Example}
\theoremstyle{definition}
\newtheorem{definition}[theorem]{Definition}
\theoremstyle{remark}
\newtheorem{remark}[theorem]{Remark}
\numberwithin{equation}{section}

\newcommand{\K}{\mathbb K}

\begin{document}

\title[ Ternary $q$-Virasoro-Witt  Hom-Nambu-Lie  algebras]
{ Ternary  $q$-Virasoro-Witt  Hom-Nambu-Lie  algebras}%
\author{F. Ammar, A. Makhlouf and S. Silvestrov}%
\address{Abdenacer Makhlouf  Universit\'{e} de Haute Alsace,  Laboratoire de Math\'{e}matiques, Informatique et Applications,
4, rue des Fr\`{e}res Lumi\`{e}re F-68093 Mulhouse, France}%
\email{Abdenacer.Makhlouf@uha.fr}
\address{Sergei Silvestrov, Centre for Mathematical Sciences,  Lund University, Box
   118, SE-221 00 Lund, Sweden}
\email{sergei.silvestrov@math.lth.se}
\address{Faouzi Ammar, Universit{\'{e}} de Sfax, Facult{\'{e}} des Sciences, B.P.
1171, 3000 Sfax, Tunisia} \email{Faouzi.Ammar@rnn.fss.tn}

\thanks {This work was partially supported by the Swedish Research Council and the SIDA Foundation, the Crafoord Foundation, the Swedish Foundation for International Cooperation in Research and
Higher Education (STINT), the Royal Swedish Academy of Sciences, Lund University, Universit\'{e} de Haute Alsace in Mulhouse and University of Sfax.}

\date{}
%
\begin{abstract}
In this paper we construct ternary $q$-Virasoro-Witt   algebras which  $q$-deform  the ternary Virasoro-Witt
algebras constructed  by Curtright, Fairlie and Zachos using $su(1,1)$ enveloping algebra techniques.  The ternary Virasoro-Witt
algebras constructed by Curtright, Fairlie and Zachos depend on a parameter and are not Nambu-Lie algebras for all but finitely many values of this parameter.
For the parameter values for which the  ternary Virasoro-Witt
algebras are Nambu-Lie, the corresponding ternary $q$-Virasoro-Witt
algebras constructed in this article are also Hom-Nambu-Lie because they are obtained from the ternary Nambu-Lie algebras using the composition method. For other parameter values this composition method does not yield Hom-Nambu Lie algebra structure for $q$-Virasoro-Witt algebras. We show however, using a different construction, that the ternary Virasoro-Witt algebras of Curtright, Fairlie and Zachos, as well as the general ternary $q$-Virasoro-Witt algebras we construct,
carry a structure of ternary Hom-Nambu-Lie algebra for all values of the involved parameters.
\end{abstract}
\maketitle

\section*{Introduction}
Lie algebras and Poisson algebras have played an extremely
important role in physics for a long time. Their generalizations,
known as $n$-Lie algebras and ``Nambu algebras''
\cite{n:generalizedmech,f:nliealgebras,Takhtajan} also arise naturally in physics and have recently been studied in the context of ``M-branes'' \cite{BL2007,HoppeJMalgNambumech}. It turns out that in the dynamic study of strings and M-branes appears naturally an algebra with ternary multiplication  called Bagger-Lambert algebra. It was used in \cite{BL2007} as one of the main ingredients in the construction of a new type of supersymmetric gauge theory that is consistent with all the symmetries
expected of a multiple M2-brane theory: 16 supersymmetries, conformal invariance, and an SO(8) R-symmetry that acts on the eight transverse scalars. Hundreds of papers are dedicated to Bagger-Lambert algebra by now. Other applications of Nambu algebras to M-branes, quantization of Nambu mechanics, volume preserving diffeomorphisms, integrable systems and related generalization of Lax equation have been considered in \cite{HoppeJMalgNambumech}.

A long-standing problem related to Nambu algebras is their
quantization. For Poisson algebras, the problem of finding an operator
algebra where the commutator Lie algebra corresponds to the Poisson
algebra is a well-studied problem. For higher order algebras much less
is known and the corresponding problem seems to be hard.
Nambu-Lie algebra is defined in general by $n$-ary multi-linear multiplication which is skew symmetric (see \eqref{trilskewsym} for $n=3$) and satisfies an identity extending the Jacobi identity for the Lie algebras. For $n=3$ this identity is
\begin{equation*}\label{TernaryNambuIdentityOrdinary}
      [ x_{1},x_{ 2}, [x_{3},x_{4},x_{5}]] =
      [ [x_1,x_2,x_3 ],x_{4},x_{5}]
      + [ x_3,[ x_1,x_2,x_{4}] , x_{5}]+
      [x_3,x_{4},[x_1,x_2,x_{5}]].
  \end{equation*}
In Nambu-Lie algebras, the additional freedom in comparison with Lie algebras is  mainly limited to extra arguments in the multi-linear multiplication. The identities of Nambu-Lie algebras are also closely resembling the identities for Lie algebras. As a result, there is a close similarity between Lie algebras and  Nambu-Lie algebras in their appearances in connection to other algebraic and analytic structures and in the extent of their applicability.
Thus it is not surprising that it becomes unclear how to associate in meaningful ways ordinary Nambu-Lie algebras with the important in physics generalizations and quantum deformations of Lie algebras when typically the ordinary skew-symmetry and Jacobi identities of Lie algebras are violated. However, if the class of Nambu-Lie algebras is extended with enough extra structure beyond just adding more arguments in multi-linear multiplication, the natural ways of association of such multi-linear algebraic structures with generalizations and quantum deformations of Lie algebras may become feasible. Hom-Nambu-Lie algebras are defined by a similar but more general  identity than that of Nambu-Lie algebras involving some additional linear maps (see Definition \ref{def:hom-Nambu-Liealg}). These linear maps twisting or deforming the main identities introduce substantial new freedom in the structure allowing to consider Hom-Nambu-Lie algebras as deformations of Nambu-Lie algebras ($n$-Lie algebras). The extra freedom built into the structure of Hom-Nambu-Lie algebras may provide a path to quantization beyond what is possible for ordinary Nambu-Lie algebras.
All this gives also important motivation for investigation
of mathematical concepts and structures such as Leibniz $n$-ary algebras \cite{CassasLodayPirashvili,f:nliealgebras} and their modifications and extensions, as well as Hom-algebra extensions of Poisson algebras \cite{HomDeform}. For discussion of physical applications of these and related algebraic structures to models for elementary particles, and unification problems for interactions see \cite{Kerner7,Kerner,Kerner2,Kerner4,Kerner6}.

Generalizations of $n$-ary algebras of Lie type and associative type by twisting the identities using linear maps have been introduced in \cite{ams:gennambu}.
These generalizations include $n$-ary
Hom-algebra structures generalizing the $n$-ary algebras of Lie type
such as $n$-ary Nambu algebras, $n$-ary Nambu-Lie algebras and
$n$-ary Lie algebras, and $n$-ary algebras of associative type
such as $n$-ary totally associative and $n$-ary partially
associative algebras.
The general Hom-algebra structures arose first in connection to
quasi-deformation and discretizations of Lie algebras of vector
fields. These quasi-deformations lead to quasi-Lie algebras, a
generalized Lie algebra structure in which the skew-symmetry and
Jacobi conditions are twisted.  The first examples were concerned
with $q$-deformations of the Witt and Virasoro algebras (see for
example
\cite{AizawaSaito,ChaiElinPop,ChaiKuLukPopPresn,ChaiKuLuk,ChaiPopPres,CurtrZachos1,
  DaskaloyannisGendefVir,Hu,LiuKQuantumCentExt,LiuKQCharQuantWittAlg,LiuKQPhDthesis}).  Motivated by these and new examples
arising as applications of the general quasi-deformation
construction in \cite{HLS,LS1,LS2} on the one hand, and the desire
to be able to treat within the same framework such well-known
generalizations of Lie algebras as the color and  Lie superalgebras
on the other hand, quasi-Lie algebras and their subclasses of
quasi-Hom-Lie algebras and Hom-Lie algebras were introduced in
\cite{HLS,LS1,LS2,LS3}. In Hom-Lie algebras
skew-symmetry is untwisted, whereas the Jacobi identity is twisted
by a single linear map and contains three terms as for Lie algebras,
reducing to ordinary Lie algebras when the linear twisting map is
the identity map.  Hom-associative algebras replacing associative
algebras in the context of Hom-Lie algebras and also more general
classes of Hom-Lie admissible algebras, $G$-Hom-associative
algebras, where introduced in \cite{MS}.  The first steps in the
construction of universal enveloping algebras for Hom-Lie algebras
have been made in \cite{Yau:EnvLieAlg}.
Formal deformations and elements of
(co-)homology for Hom-Lie algebras have been studied in
\cite{HomDeform,Yau:HomolHomLie}, whereas dual structures such as Hom-coalgebras, Hom-bialgebras and Hom-Hopf algebras appeared first in
\cite{HomHopf,HomAlgHomCoalg} and further investigated in
\cite{IGSC,Yau:HomBial}.

The aim of this paper is to construct ternary $q$-Virasoro-Witt   algebras which $q$-deform  the ternary Virasoro-Witt algebras introduced   by Curtright, Fairlie and Zachos  \cite{CurtrFairlZachos} in such a way that they carry a structure of  ternary Hom-Nambu-Lie algebras for all values of involved parameters.
The ternary Virasoro-Witt
algebras constructed by Curtright, Fairlie and Zachos are  ternary Nambu-Lie algebras  only for two values of the parameter $z=\pm 2 i$. We show in this paper that for these values the corresponding  ternary $q$-Virasoro-Witt
algebras  are   ternary Hom-Nambu-Lie algebras as they are obtained from those ternary Nambu-Lie algebras using the composition with a homomorphism according to a general method described in \cite{ams:gennambu}. For the other values of the parameter $z$ this does not automatically yield the ternary Hom-Nambu-Lie structure since the algebras one starts with are not Nambu-Lie. We prove however that these ternary algebras of Curtright, Fairlie and Zachos, as well as  the general ternary $q$-Virasoro-Witt   algebras we construct, carry a structure of ternary Hom-Nambu-Lie algebra in other ways, for any values of the parameters.

The paper is organized as follows. In Section \ref{VWalgebras}, we summarize  the
definition of ternary Virasoro-Witt algebras constructed through the
use of $su(1,1)$ enveloping algebra techniques by Curtright, Fairlie
and Zachos and also describe a connection to  the ternary algebra introduced by T. A.  Larsson \cite{LarssonTA}.
In Section \ref{sec:NambuLieHomNambuLie}, we recall the definitions and properties of ternary
Nambu-Lie algebras and ternary Hom-Nambu-Lie algebras including also the composition method
of construction of ternary Hom-Nambu-Lie algebras from
ternary Nambu-Lie algebras. In Section \ref{HomNambuVW},
we describe a family of homomorphisms of the a ternary Nambu-Lie algebras
of Curtright, Fairlie
and Zachos and use this and the composition method to define
$q$-Virasoro-Witt algebras. Section \ref{sec:HomNambuLieTernqVirWitt} is devoted
to  the structure
of ternary Virasoro-Witt algebras. We show that for any value of the
parameter they carry a Hom-Nambu-Lie structure. Also we describe all
the twisting maps of a specific natural type. Finally, in Section \ref{sec:TernHomNambuLieTernNambuLie} we discuss
in more general situation the problem of  ternary Hom-Nambu-Lie
algebras induced by ternary Nambu-Lie algebras and make some
observations about ternary q-Virasoro-Witt algebras.
Throughout this paper, $\K$ will denote an algebraically closed field of characteristic zero.

\section{Ternary Virasoro-Witt algebras}\label{VWalgebras}
In this Section, we define  ternary Virasoro-Witt algebras
constructed through the use of $su(1,1)$ enveloping algebra
techniques by Curtright, Fairlie and Zachos \cite{CurtrFairlZachos}.

\begin{definition}
 The ternary algebras given by the linear space $W$, generated by
 $\{ Q_n,R_n\}_{n\in\mathbb{Z}}$, with ternary trilinear multiplication defined by
skewsymmetric ternary brackets
\begin{eqnarray}\label{CFZbracket1}
[ Q_k,Q_m,Q_n]&=& (k-m)(m-n)(k-n) R_{k+m+n}\\
\label{CFZbracket2} [Q_k,Q_m,R_n]&=& (k-m)(Q_{k+m+n}+ z n R_{k+m+n})\\
\label{CFZbracket3} [Q_k,R_m,R_n]&=& (n-m) R_{k+m+n}\\
\label{CFZbracket4} [R_k,R_m,R_n]&=& 0
\end{eqnarray}
are called \emph{ternary Virasoro-Witt algebras}.
\end{definition}
\begin{remark}
Actually the previous ternary algebra is a ternary Nambu-Lie algebra
only in the cases $z=\pm 2 i$. In Section \ref{HomNambuVW}, we
show that they carry for any $z$ a structure of ternary
Hom-Nambu-Lie algebras.
\end{remark}

Let us recall the connection to the ternary algebras considered by T. A. Larsson  in \cite{LarssonTA}.   He considered the operators
\begin{equation}
  E_m = e^{i m x}, \quad  \quad
  L_m = e^{i m x}(-i \frac{d}{dx}+\lambda m).
\end{equation}
Here $E_m$ should be understood as operator of multiplication by
$e^{imx}$ on a suitable space of functions.

These operators satisfy the following commutation relations
$$[L_m,L_n]=(n-m)L_{m+n},\quad [E_m,E_n]=0,\quad
[L_m,E_n]=n E_{m+n}.
$$
where the braket means the usual commutator of linear operators $[A,B]=AB-BA$.

In any associative algebra with multiplication $"\cdot"$,
define the ternary multi-linear multiplication (bracket) by
\begin{equation}\label{NambuCommutator1}
\begin{array}{l}
[x,y,z] = x\cdot [y,z]+y\cdot [z,x]+z\cdot [x,y]\\
= x\cdot(y\cdot z)-x\cdot(z\cdot y)+y\cdot(z\cdot x)-y\cdot(x\cdot z)+z\cdot(x\cdot y)-z\cdot(y\cdot x)
\end{array}
\end{equation}
where $[\cdot,\cdot]$ denotes the binary commutator bracket for its corresponding Lie algebra ($[a,b]=a\cdot b - b\cdot a$).
Using \eqref{NambuCommutator1}, one gets
\begin{eqnarray}\label{Lars1}
 [L_k,L_m,L_n]&=& (\lambda-\lambda ^2)(k-m)(m-n)(n-k) E_{k+m+n}\\
\label{Lars2} \  [L_k,L_m,E_n]&=& (m-k)(L_{k+m+n}+(1- 2\lambda) n E_{k+m+n})\\
\label{Lars3} \  [L_k,E_m,E_n]&=& (m-n) E_{k+m+n}\\
\label{Lars4}\  [ E_k,E_m,E_n]&=& 0
\end{eqnarray}

 The brackets
\eqref{CFZbracket1}-\eqref{CFZbracket4} are obtained by taking, for $\lambda\neq 0$ and $\lambda\neq 1$
$$L_m=-(\lambda-\lambda ^2)^{1/4} Q_m, \quad
E_m=(\lambda-\lambda ^2)^{-1/4} R_m, \quad z=-\frac{1-2
\lambda}{(\lambda-\lambda ^2)^{1/2}}.
$$
However, exactly the ternary Nambu-Lie algebras of Curtright, Fairlie and Zachos ($z=\pm 2i$) cannot be obtained from \eqref{Lars1}-\eqref{Lars4} in this way because $z=-\frac{1-2
\lambda}{(\lambda-\lambda ^2)^{1/2}}$  leads to $(4+z^2)\lambda^2-(4+z^2)\lambda +1=0$ which has no solution for $z=\pm 2i$.

For $z\neq \pm 2i$, using that change of generators, we obtain
a realization of ternary Virasoro-Witt algebras of Curtright,
Fairlie and Zachos by the following operators
\begin{eqnarray*}
  Q_m &=&-(\lambda-\lambda ^2)^{-1/4} e^{i m x} \\
  R_m &=& e^{i m x}(-i (\lambda-\lambda ^2)^{1/4}  \frac{d}{dx}+\lambda (\lambda-\lambda ^2)^{1/4}  m).
\end{eqnarray*}
Note that these operators also satisfy  the following commutation relations
\begin{gather*}
[Q_m,Q_n]=-(\lambda-\lambda ^2)^{-1/4} (n-m)Q_{m+n}, \quad [R_m,R_n]=0, \\
[Q_m,R_n]=-(\lambda-\lambda ^2)^{-1/4} n R_{m+n}.
\end{gather*}
which also lead to \eqref{CFZbracket1}-\eqref{CFZbracket4} using \eqref{NambuCommutator1}.
In the presentation of Larsson appeared also a second order differential operators $S_m = e^{i m x}(-i \frac{d}{dx}+\lambda m)^2$. These operators can be expressed in several ways using operators $E_m$ and $L_m$. For example $S_m = L_m E_{-m} L_m$.
They are not needed to recover the ternary  Virasoro-Witt algebras.

\begin{remark}
One may notice that the ternary bracket  \eqref{NambuCommutator1}
does not lead automatically to ternary Nambu-Lie algebra when
starting from an associative algebra and the corresponding Lie
algebra given by the binary commutators. Another way to construct ternary Nambu Lie algebras starting from a Lie algebra and a trace function could be found in \cite{Awata}. This construction have been extended to ternary Hom-Nambu-Lie algebras in \cite{AMSternaryCommutator2009}.
\end{remark}

\section{Nambu-Lie ternary algebras and Hom-Nambu-Lie ternary algebras} \label{sec:NambuLieHomNambuLie}

In this Section, we recall the definitions of introduced in
\cite{ams:gennambu} ternary Hom-Nambu algebras and ternary
Hom-Nambu-Lie algebras generalizing the usual  ternary  Nambu algebras and ternary   Nambu-Lie algebras (called also Filippov ternary
algebras).
\begin{definition}
A \emph{ternary  Hom-Nambu algebra} is a triple $(V,[\cdot,\cdot,\cdot],\alpha)$,
consisting of a vector space $V$, a trilinear map $[\cdot,\cdot,\cdot]:V\times V\times V\rightarrow V$
and a pair $\alpha=(\alpha_{1},\alpha_{2})$ of linear maps
$\alpha_{i}:V\rightarrow V$, $i=1,2$ satisfying
\begin{equation} \label{TernaryNambuIdentity}
\begin{array}{l}
[ \alpha_1(x_{1}),\alpha_{2}(x_{ 2}), [x_{3},x_{4},x_{5}]]=
 [ [x_1,x_2,x_3 ],\alpha_1(x_{4}),\alpha_{2}(x_{5})] + \\[0cm]
\quad \quad \quad \quad  [ \alpha_{1} (x_3),[ x_1,x_2,x_{4}] , \alpha_{2}(x_{5})]+
[\alpha_1(x_3),\alpha_{2}(x_{4}),[x_1,x_2,x_{5}]].
\end{array}
\end{equation}
\end{definition}
The identity (\ref{TernaryNambuIdentity}) is called Hom-Nambu
identity.

\begin{remark}
When the maps $(\alpha_{i})_{i=1,2}$ are all identity maps,  one
recovers the classical ternary Nambu algebras and the identity
 \eqref{TernaryNambuIdentity} is called  Nambu identity. The Nambu identity
  is also known
 as the fundamental identity or Filippov identity
\cite{f:nliealgebras, n:generalizedmech, Takhtajan}.
\end{remark}

\begin{remark}
Let $(V,[\cdot,\cdot,\cdot],\alpha)$ be a ternary Hom-Nambu algebra
 where $\alpha=(\alpha_{1},\alpha_{2})$. Let $x=(x_1,x_{2})\in
V\times V$, $\alpha (x)=(\alpha_1(x_1),\alpha_{2}(x_{2}))\in V\times
V$ and $y\in V$. Let $L_x$ be a linear map on $V$,
defined   by  $$L_{x}(y)=[x_{1},x_{2},y].$$
Then the Hom-Nambu identity is
\begin{gather*}
L_{\alpha (x)}( [x_{3},x_{4},x_{5}])= [L_{x}(x_{3}),\alpha_1(x_{4}),
 \alpha_{2}(x_{5})]+ \\[0cm]
[\alpha_1(x_{3}), L_{x}(x_{4}),
\alpha_{2}(x_{5})]+[\alpha_1(x_{3}), \alpha_{2}(x_{4}, L_{x}(x_{5})].
\end{gather*}%
\end{remark}

For the rest of the paper, we will use the following handy notation
for the difference between the left hand side and the right hand
side of the Nambu identity for the ternary Nambu algebras:
\begin{eqnarray}\label{LHSminusRHSTernaryNambuIdentity}
FI(x_{1},x_{2},x_{3},x_{4},x_{5})&=[ x_{1},x_{ 2},
[x_{3},x_{4},x_{5}]]- [ [x_1,x_2,x_3 ],x_{4},x_{5}]\\ \ & - [ x_3,[
x_1,x_2,x_{4}] , x_{5}]- [x_3,x_{4},[x_1,x_2,x_{5}]], \nonumber
\end{eqnarray}
and analogous notation for the Hom-Nambu identity of Hom-Nambu algebras
\begin{equation} \label{LHSminusRHSTernaryHomNambuIdentity}
\begin{array}{l}
HFI(x_{1},x_{2},x_{3},x_{4},x_{5})  =[ \alpha_1
(x_{1}),\alpha_{2}(x_{ 2}), [x_{3},x_{4},x_{5}]] \\
 \quad \quad - [ [x_1,x_2,x_3],\alpha_1 (x_{4}),\alpha_{2}(x_{5})] - [ \alpha_1 (x_3),[
x_1,x_2,x_{4}] , \alpha_{2}(x_{5})] \\
 \quad \quad \quad \quad \quad \quad - [\alpha_1
(x_3),\alpha_{2}(x_{4}),[x_1,x_2,x_{5}]].
\end{array}
\end{equation}

\begin{definition} \label{def:hom-Nambu-Liealg}
A ternary  Hom-Nambu algebra $(V,[\cdot,\cdot,\cdot],\alpha)$, where
$\alpha=(\alpha_{1},\alpha_{2})$, is called \emph{ternary
Hom-Nambu-Lie algebra} if the bracket is skew-symmetric that is
\begin{equation}  \label{trilskewsym}
[x_{\sigma (1)}, x_{\sigma (2)}, x_{\sigma (3)}]=Sgn(\sigma )[x_{1},
x_{2}, x_{3}],\quad \forall \sigma \in \mathcal{S}_{3}\text{ and
}\forall x_{1}, x_{2}, x_{3}\in V
\end{equation}%
where $\mathcal{S}_{3}$ stands for the permutation group on $3$
elements.

In particular, the skewsymmetric ternary Nambu algebras are called ternary Nambu-Lie algebras
\end{definition}

The morphisms of ternary Hom-Nambu algebras  are defined in the
natural way. It should be pointed out however that the morphisms
should intertwine not only the ternary products but also the
twisting linear maps.

\begin{definition}
Let $(V,[\cdot,\cdot,\cdot],\alpha)$ and
$(V',[\cdot,\cdot,\cdot]',\alpha')$ be two $n$-ary Hom-Nambu
algebras (resp. $n$-ary Hom-Nambu-Lie algebras) where
$\alpha=(\alpha_{1},\alpha_{2})$ and
$\alpha'=(\alpha'_{1},\alpha'_{2})$. A linear map $\rho:
V\rightarrow V'$ is a ternary Hom-Nambu algebras morphism (resp.
ternary Hom-Nambu-Lie algebras morphism) if it satisfies
\begin{eqnarray*}\rho ([x_{1},
x_{2}, x_{3}])&=&
[\rho (x_{1}),\rho (x_{2}),\rho (x_{3})]'\\
\rho \circ \alpha_i&=&\alpha'_i\circ \rho \quad \forall i=1,2.
\end{eqnarray*}
\end{definition}

The following theorem, given in \cite{ams:gennambu} for $n$-ary algebras
of Lie type, provides a way to construct ternary Hom-Nambu algebra
(resp. ternary Hom-Nambu-Lie algebra) starting from ternary Nambu
algebra (resp. ternary Nambu-Lie algebra) and a ternary algebras
endomorphism.
\begin{theorem}[\cite{ams:gennambu}]\label{ThmConstrHomNambuLie}
Let $(V,[\cdot,\cdot,\cdot])$ be a ternary Nambu algebra (resp.
ternary Nambu-Lie algebra)  and let $\rho : V\rightarrow V$ be a
ternary Nambu (resp. ternary Nambu-Lie) algebras endomorphism.

We set $[\cdot,\cdot,\cdot]_\rho=\rho\circ [\cdot,\cdot,\cdot]$ and
$\widetilde{\rho}=(\rho,\rho)$. Then
$(V,[\cdot,\cdot,\cdot]_\rho,\widetilde{\rho})$
  is a ternary Hom-Nambu algebra (resp.  ternary Hom-Nambu-Lie algebra).

Moreover, suppose that  $(V',[\cdot,\cdot,\cdot]_\rho')$ is another ternary
Nambu algebra (resp.  ternary Nambu-Lie algebra)  and
$\rho ' : V'\rightarrow V'$ is a ternary Nambu
(resp.  ternary Nambu-Lie) algebra endomorphism. If $f:V\rightarrow
V'$ is a ternary Nambu algebras morphism (resp. ternary Nambu-Lie
algebras morphism) that satisfies $f\circ\rho=\rho'\circ f$ then
$$f:(V,[\cdot,\cdot,\cdot]_\rho,\widetilde{\rho})\longrightarrow
(V',[\cdot,\cdot,\cdot]'_{\rho '},\widetilde{\rho '})
$$
is a ternary Hom-Nambu algebras morphism (resp.  ternary
Hom-Nambu-Lie algebras morphism).
\end{theorem}

\begin{example}
 An algebra $V$ consisting of polynomials or possibly of other differentiable functions in $3$ variables $x_{1},x_{2},x_{3},$ equipped with well-defined bracket multiplication given by the functional jacobian
 $J(f)=(\frac{\partial f_i}{\partial x_j})_{1\leq i,j \leq 3}$:
\begin{equation} \label{JacternaryNambuLie}
\lbrack f_{1},f_{2},f_{3}\rbrack=\det \left(
\begin{array}{ccc}
\frac{\partial f_{1}}{\partial x_{1}} & \frac{\partial
f_{1}}{\partial x_{2}} &
\frac{\partial f_{1}}{\partial x_{3}} \\
\frac{\partial f_{2}}{\partial x_{1}} & \frac{\partial
f_{2}}{\partial x_{2}} &
\frac{\partial f_{2}}{\partial x_{3}} \\
\frac{\partial f_{3}}{\partial x_{1}} & \frac{\partial
f_{3}}{\partial x_{2}} &
\frac{\partial f_{3}}{\partial x_{3}}%
\end{array}%
\right),
\end{equation}%
is a ternary  Nambu-Lie algebra. By considering a ternary Nambu-Lie
algebra endomorphism of such algebra, we construct a ternary
Hom-Nambu-Lie algebra. Let $\gamma(x_1,x_2,x_3)$ be a polynomial or
more general differentiable transformation of three variables
mapping elements of $V$ to elements of $V$ and such that $\det J
(\gamma) = 1$. Let $\rho_\gamma: V \mapsto V$ be the composition
transformation defined by $f\mapsto f\circ \gamma $ for any $f \in
V$. By the general chain rule for composition of transformations of
several variables,
\begin{eqnarray*}
J(\rho_\gamma (f)) = J(f\circ \gamma)=(J(f)\circ \gamma) J(\gamma)=\rho_\gamma (J(f))J(\gamma), \\
\det J(\rho_\gamma (f)) = \det (J(f)\circ \gamma) \det J(\gamma)=
\det \rho_\gamma (J(f))\det J(\gamma).
\end{eqnarray*}
Hence, for any transformation $\gamma$ with $\det J(\gamma)=1$, the
composition transformation $\rho_\gamma$ defines an endomorphism of
the ternary  Nambu-Lie algebra with ternary product
\eqref{JacternaryNambuLie}. Therefore, by Theorem
\ref{ThmConstrHomNambuLie}, for any such transformation $\gamma$,
the triple
$$\left(V,\lbrack \cdot,\cdot,\cdot \rbrack_\gamma=\rho_\gamma\circ \lbrack \cdot,\cdot,\cdot \rbrack,\widetilde{\rho_\gamma}\right)$$
is a ternary Hom-Nambu-Lie algebra.
\end{example}
\begin{remark}
One may expect that a ternary generalization of the binary Witt
algebra defined on the generators $Q_n$ by the bracket
$[Q_n,Q_m]=(m-n)Q_{n+m}$ could be given by a ternary bracket defined
on the generators $Q_n$ by
$$[ Q_k,Q_m,Q_n]= (m-k)(n-m)(n-k) Q_{k+m+n}
$$
It turns out that in this case the Nambu identity is not satisfied
\begin{multline*}FI(Q_u,Q_v,Q_k,Q_m,Q_n)= -2 (k-m) (k-n) (m-n) (-u+v)
\\ (-n u+u^2+m (n-u-v) +k (m+n-u-v)-nv+uv+v^2) Q_{k+m+n+u+v}
\end{multline*}
Moreover, it is not a ternary Hom-Nambu-Lie algebra for any
nontrivial linear maps of the form
\begin{equation} \alpha_1(Q_n)=
a_n Q_n \quad \alpha_2(Q_n)= b_n Q_n.
\end{equation}
\end{remark}

\section{Ternary $q$-Virasoro-Witt algebras} \label{HomNambuVW}
In this Section, we describe a family of homomorphisms of the ternary Nambu-Lie algebras
of Curtright, Fairlie
and Zachos and use this and the composition method to define
$q$-Virasoro-Witt algebras.
We also use this and Theorem \ref{ThmConstrHomNambuLie}
to construct Hom-Nambu-Lie algebras associated to
those two ternary Virasoro-Witt algebras defined by
\eqref{CFZbracket1}-\eqref{CFZbracket4} which are ternary Nambu-Lie
algebras ($z=\pm 2 i$).

Let us  provide first a class of algebra homomorphisms of ternary
Virasoro-Witt algebras. We consider the linear maps $f$ defined on
the generators by
\begin{eqnarray}
f (Q_n)= a_n Q_n,\quad f (R_n)= b_n  Q_n.
\end{eqnarray}
satisfying for any generators $X_n,Y_m,Z_n$ of the ternary
Virasoro-Witt algebra
\begin{equation}
f ([X_n,Y_m,Z_n])=[f (X_n),f (Y_m),f (Z_n)].
\end{equation}
The previous identity applied to the triples $(Q_k,Q_m,Q_n)$,
$(Q_k,Q_m,R_n)$, $(Q_k,R_m,R_n)$ and $(R_k,R_m,R_n)$, for any
$k,m,n\in \mathbb{Z}$ lead to the system of equations
 \begin{align}
b_{k+m+n}&=a_{k}a_m a_n\quad \text {for } k\neq m\neq n\neq k\\
a_{k+m+n}&=a_{k} a_m b_n\quad \text {for } k\neq m\neq n\\
b_{k+m+n}&=a_{k} a_m b_n\quad \text {for } k\neq m\neq n\neq 0\\
b_{k+m+n}&=a_{k} b_m b_n\quad \text {for }  m\neq n\
\end{align}
which has as solution $a_n=b_n=q^n$, with $q\in \K$.
\begin{proposition} \label{prop:endoternVirWit}
The linear map  defined on the ternary Virasoro-Witt algebra by
\begin{eqnarray}
f (Q_n)=  q^n Q_n,\quad f (R_n)=  q^n  Q_n,\quad \text{ for } q\in\K
\end{eqnarray}
is a ternary algebra endomorphism.
\end{proposition}

Theorem \ref{ThmConstrHomNambuLie} and the ternary
Virasoro-Witt algebra endomorphism described in
Proposition \ref{prop:endoternVirWit} lead to the following result.

\begin{theorem}\label{ThnHomNambuLieVW}
Let $W$ be the linear space generated by
 $\{ Q_n,R_n\}_{n\in\mathbb{Z}}$. Let
skewsymmetric trilinear bracket multiplication $[\cdot, \cdot, \cdot] W^3 \rightarrow W$ be defined on the generators of $W$ by
\begin{eqnarray}\label{qVW1}
[ Q_k,Q_m,Q_n]_q&=& q^{k+m+n}(k-m)(m-n)(k-n) R_{k+m+n}\\
\label{qVW2} [Q_k,Q_m,R_n]_q&=& q^{k+m+n}(k-m)(Q_{k+m+n}+ z n R_{k+m+n})\\
\label{qVW3} [Q_k,R_m,R_n]_q&=& q^{k+m+n}(n-m) R_{k+m+n}\\
\label{qVW4} [R_k,R_m,R_n]_q&=& 0
\end{eqnarray}
for $q\in\K$.  Let $\alpha_1$ and $\alpha_2$ be the linear maps on $W$  given by
\begin{equation}
\alpha_1(Q_n)= q^n Q_n \quad \alpha_1(R_n)= q^n R_n,
\end{equation}
\begin{equation}
\alpha_2(Q_n)= q^n Q_n \quad \alpha_2(R_n)= q^n R_n.
\end{equation}
Then $(W,[~,~,~]_q, (\alpha_1,\alpha_2))$ are ternary Hom-Nambu-Lie
algebras  if $z=\pm 2 i$.
\end{theorem}

We call $(W,[~,~,~]_q)$ for any values of parameters $z$ a ternary $q$-Virasoro-Witt algebras.

\begin{remark}
In Theorem \ref{ThnHomNambuLieVW}, for $q=1$ we recover the ternary
Nambu-Lie algebras obtained by Curtright, Fairlie and Zachos. We
assume in the sequel that $q\neq 0$, since otherwise we have a zero
ternary algebra.
\end{remark}

\begin{remark}
The ternary Hom-Nambu-Lie algebras are actually not ternary
Nambu-Lie algebras for $q\neq 1$.
Indeed,
\begin{align*}
& FI(Q_u,Q_v,Q_k,Q_m,Q_n)=\\
& q^{k + m + n + u + v} (u - v) (q^{k + m + n}(k - m) (k - n) (m - n)+  \\
& q^{m + u + v} (k - n)(m - u) (m - v)+ \\
& q^{k + u + v} (n-m )(u-k ) (v-k )+ \\
& q^{n + u + v} (m-k )(u-n ) (v-n )) Q_{k + m + n + u + v}+ \\
& q^{k + m + n + u + v} (u - v) ( (t^{k + m + n}(k - m) (k - n) (m - n) (k + m + n)   + \\
& q^{k + u + v}(n-m )  (u-k  ) (v-k ) (k + u + v) + \\
& q^{m + u + v}(k - n)  (m - u) (m - v) (m + u + v) + \\
& q^{n + u + v}(m-k )  (u-n ) (v-n) (n + u + v)) z) R_{k + m + n + u + v}
\end{align*}
\end{remark}

\section{On the Hom-Nambu-Lie structure of Ternary $q$-Virasoro-Witt algebras} \label{sec:HomNambuLieTernqVirWitt}
For $z\neq \pm 2i$ and any $q\in \mathbb{C}$, the ternary
$q$-Virasoro-Witt algebras defined by \eqref{qVW1}-\eqref{qVW4} are Hom-Nambu-Lie since they are obtained by composition method
(Theorem \ref{ThmConstrHomNambuLie}) from the ternary Virasoro-Witt algebras of Curtright, Fairlie and Zachos which are Nambu-Lie for $z=\pm 2i$. For values $z\neq \pm 2i$ this does not work since then the Virasoro-Witt algebras are not Nambu-Lie. Nevertheless, this does not exclude that these ternary algebras are Hom-Nambu-Lie in some other way.

We show in this Section that the ternary $q$-Virasoro-Witt algebras
defined by the brackets (\ref{qVW1}-\ref{qVW4}) carry,   for any
values of $z$ and $q$, a structure of ternary Hom-Nambu-Lie algebra.
We provide all the homomorphisms of a given specific natural type yielding this structure. More precisely, we consider  twisting maps $\alpha_1$ and $\alpha_2$ defined on the generators of $W$ by
\begin{eqnarray}
\alpha_1 (Q_n)&=  a^1_n  Q_n+b^1_n  R_n, \quad  \alpha_1 (R_n)&=
c^1_n Q_n+d^1_n  R_n\\
\alpha_2 (Q_n)&=  a^2_n  Q_n+b^2_n  R_n, \quad  \alpha_2 (R_n)&=
c^2_n Q_n+d^2_n  R_n
\end{eqnarray}
We assume further that the supports of the function $a^i_n, b^i_n,
c^i_n, d^i_n$, $i=1,2$ are either $0$ or all $\mathbb{Z}$. Then we
get, for any values of $z$ and $q$, the following result.
\begin{theorem}\label{thmHNL2}Let $W$ be the linear space
generated by
 $\{ Q_n,R_n\}_{n\in\mathbb{Z}}$.
Let  $\alpha_1$ and $\alpha_2$ be two linear maps on $W$ defined on
the generators by
\begin{eqnarray}
\alpha_1 (Q_n)&=  a^1_n  Q_n+b^1_n  R_n, \quad  \alpha_1 (R_n)&=
c^1_n Q_n+d^1_n  R_n\\
\alpha_2 (Q_n)&=  a^2_n  Q_n+b^2_n  R_n, \quad  \alpha_2 (R_n)&=
c^2_n Q_n+d^2_n  R_n.
\end{eqnarray}
Assume further that the supports of the function $a^i_n, b^i_n,
c^i_n, d^i_n$, $i=1,2$ are either $0$ or all $\mathbb{Z}$.

Then the ternary $q$-Virasoro-Witt algebras on $W$
defined by the brackets \eqref{qVW1}-\eqref{qVW4} are ternary Hom-Nambu-Lie algebras  with twisting maps  $\alpha_1$ and
$\alpha_2$, for any value of $z$ and $q$,  if and only if
\begin{eqnarray}
\alpha_1 (Q_n)= \beta_1 q^n  R_n,\quad \alpha_1 (R_n)= 0,\\
\alpha_2 (Q_n)= \beta_2 q^n  R_n,\quad \alpha_2(R_n)= 0.
\end{eqnarray}
where $\beta_1,\beta_2\in \K.$
\end{theorem}

\begin{proof}
The Hom-Nambu identities are equivalent to a large system of
equations for the coefficients $a^i_n, b^i_n, c^i_n, d^i_n$,
$i=1,2$. Solving this system amounts to going through a large tree
of cases. The identity $HFI
(Q_u,Q_v,Q_k,Q_m,Q_n)=0$ leads for any $u,v,k,m,n$ to the following
system of equations
\begin{equation}\label{Sec4eq1}
\begin{array}{l}
a^1_k \left(q^n a^2_m-q^m a^2_n\right)=0,\\
   \left(q^m a^1_k-q^k  a^1_m\right) a^2_n =0,\\
  q^{u+v}  a^1_k   a^2_m -q^{k+m}  a^1_u   a^2_v =0,\\
   \left(q^n  a^2_m -q^m  a^2_n \right)  b^1_k =0,\\
    a^2_n  \left(q^m  b^1_k -q^k  b^1_m \right)=0,\\
   q^{u+v}  a^2_n   b^1_k -q^{k+n}  a^2_v   b^1_u =0,\\
   q^k  a^2_n   b^1_m -q^n  a^1_k   b^2_m =0,\\
   q^n  a^2_m   b^1_k -q^k  a^1_m   b^2_n =0,\\
    \left(q^m  a^1_k -q^k  a^1_m \right)  b^2_n =0,\\
      a^1_k  \left(q^n  b^2_m-q^m b^2_n\right)=0,\\
     a^2_m b^1_k-a^1_k b^2_m =0, \\
     q^{k+n} a^2_v b^1_u-q^{u+v} a^1_k b^2_n =0,\\
     \left(-q^{u+v} a^1_m b^2_n+q^{m+n} a^1_u b^2_v\right)=0.
\end{array}
\end{equation}

The  first equation of \eqref{Sec4eq1} yields $ a^1_n=0$ or
$q^n a^2_m-q^m a^2_n=0$ which implies, since this is
true for any $n$ and $m$,  that $ a^2_n$ is of the form $a^2_n=\lambda_2 q^n$, where $\lambda_2\in \K$.

\noindent \textbf{Case 1.} $ a^1_n=0.$  The system \eqref{Sec4eq1} reduces to
\begin{equation}\label{Sec4case1}
\begin{array}{l}
a^2 _m b^1 _k = 0,\\ \left (q^n a^2 _m - t^m a^2 _n \right) b^1 _k
=0,\\ a^2 _n\left (q^m b^1 _k - q^k b^1 _m \right) = 0.
\end{array}
\end{equation}
Therefore, it follows two subcases to study $a^2 _n=0$ or $a^2 _n\neq 0$ and $ b^1
_n=0$.

\noindent \textbf{Case 1.1. $ a^1_n=0, \ a^2_n=0$}. The identity $HFI
(Q_u,Q_v,Q_k,Q_m,Q_n)=0$ is fulfilled in this case but the identity
$HFI (Q_u,R_v,Q_k,Q_m,Q_n)=0$ for any $u,v,k,m,n$ is equivalent to
the following system of equations
\begin{equation}\label{Sec4case1.1}
\begin{array}{l}
 (q^m b^1_k - q^k b^1_m) b^2_n=0\\
 b^1_k (q^n b^2_m - q^m b^2_n)=0\\
    (q^n b^1_k b^2_m - q^k b^1_m b^2_n)=0\\
   b^1_u c^2_v=0
\end{array}
\end{equation}Then, we consider the following subcases

\noindent \textbf{Case 1.1.1. $ a^1_n=0, \ a^2_n=0,\ b^1_n=0$}.

The identity
$HFI (Q_u,Q_v,R_k,Q_m,Q_n)=0$  is equivalent to
the following system of equations
\begin{equation}\label{Sec4case1.1.1}
\begin{array}{l}
 c^1_k (q^n b^2_m - q^m b^2_n)=0\\
   b^2_n c^1_m=0
\end{array}
\end{equation}
In case $b^2_n=0$,  the identity
$HFI (Q_u,R_v,R_k,R_m,Q_n)=0$  is equivalent to the equations
$$c^1_n c^2_m=0,\ c^2_n d^1_m=0,\ c^1_n d^2_m=0,\ d^1_n d^2_m=0.$$
This case leads to $\alpha_1$ or $\alpha_2$ identically trivial.

In case $b^2_n\neq 0$ we have  $c^1_n= 0$. The remaining identities lead to $\alpha_1$  identically trivial.

\noindent \textbf{Case 1.1.2. $ a^1_n=0, \ a^2_n=0,\ b^1_n\neq 0$}. According to the second equation of the system \eqref{Sec4case1.1} we should consider $b^2_n=\beta_2
q^n$, where $\beta_2$ is any scalar in $\K$.
The identity
$HFI (Q_u,R_v,Q_k,Q_m,Q_n)=0$  translates to the
equations
$
q^n b^1_m - q^m b^1_n=0, \
   b^1_n c^2_m=0.
$
Thus, $b^1_n=\beta_1
q^n$, where $\beta_1$ is any scalar in $\K$ and $c^2_n=0.$

The remaining identities imply that $c^1_n=0,\ d^1_n=0,\ d^2_n=0.$ Therefore, we get only
one solution, that is
\begin{eqnarray}
\alpha_1 (Q_n)= \beta_1 q^n  R_n,\quad \alpha_1 (R_n)= 0,\\
\alpha_2 (Q_n)= \beta_2 q^n  R_n,\quad \alpha_2(R_n)= 0.
\end{eqnarray}

\noindent \textbf{Case 1.2. $ a^1_n=0, \ a^2_n\neq 0$}.
 According to  the system \eqref{Sec4case1} we have $b^1_n=0.$

The identity
$HFI (Q_u,Q_v,R_k,Q_m,Q_n)=0$  implies
$$
   a^2_n c^1_m=0,\ b^2_n c^1_m=0,\ a^2_n d^1_m=0.
$$ Since $a^2_n\neq 0$ then $
   c^1_m=0$ and   $d^1_m=0.
$ Thus $\alpha_1$ is identically trivial.

\noindent \textbf{Case 2. $ a^1_n\neq0$}. According to \eqref{Sec4eq1} we have $a^2_n=\lambda_2
q^n$ and $b^2_n=\beta_2
q^n.$

The identity
$HFI (Q_u,R_v,Q_k,Q_m,Q_n)=0$  implies $ a^1_n c^2_m=0$. Since  $ a^1_n\neq0$ then  $ c^2_m=0$.

We consider in the following two subcases $\lambda_2=0$ and $\lambda_2\neq0$.

\noindent \textbf{Case 2.1. $ a^1_n\neq0, a^2_n=0$}.
The
identities $HFI (Q_u,Q_v,Q_k,Q_m,Q_n)=0$ and $HFI
(Q_u,R_v,Q_k,Q_m,Q_n)=0$ implies $b^2_n=0, \ d^2_n=0$.
Therefore $\alpha_2\equiv 0$.

\noindent \textbf{Case 2.2. $ a^1_n\neq0, a^2_n\neq0$}.
The
identity $HFI (Q_u,Q_v,R_k,Q_m,Q_n)=0$ leads to  $ c^1_n=0$ and $a^1_n=\lambda_1
q^n$ with $\lambda_1\neq 0$.
The
identity $HFI (Q_u,Q_v,Q_k,Q_m,Q_n)=0$ leads to  $b^1_n=\beta_1
q^n$ and $\beta_2=\frac{\lambda_2
\beta_1}{\lambda_1}.$
Using $HFI (Q_u,R_v,Q_k,Q_m,Q_n)=0$ and $HFI (Q_u,Q_v,R_k,Q_m,Q_n)=0$ we obtain $d^2_n=\lambda_2 t^n$,
 $d^1_n=\lambda_1 t^n$ and $\beta_1 =0$. Therefore $\beta_2 =0$.
In this case, the Hom-Nambu-Lie identities reduce to
 $ \lambda_1 \lambda_2(z^2+4)=0,$
which is impossible for any $z\neq \pm 2i$, since $\lambda_1
\lambda_2 \neq 0$.
In fact, this case corresponds to Theorem \ref{ThnHomNambuLieVW} where $z^2+4=0$ and
\begin{eqnarray}
\alpha_1 (Q_n)= \lambda_1 q^n  Q_n,\quad \alpha_1 (R_n)= \lambda_1 q^n  R_n\\
\alpha_2 (Q_n)= \lambda_2 q^n  Q_n,\quad \alpha_2 (R_n)= \lambda_2 q^n
R_n.
\end{eqnarray}
\end{proof}
In the particular case $q=1$, we obtain the following corollary.
\begin{proposition}
Let  $\alpha_1$ and $\alpha_2$ be two linear maps defined  on the
generators of ternary  Virasoro-Witt algebras by
\begin{eqnarray}
\alpha_1 (Q_n)&=  a^1_n  Q_n+b^1_n  R_n, \quad  \alpha_1 (R_n)&=
c^1_n Q_n+d^1_n  R_n\\
\alpha_2 (Q_n)&=  a^2_n  Q_n+b^2_n  R_n, \quad  \alpha_2 (R_n)&=
c^2_n Q_n+d^2_n  R_n.
\end{eqnarray}
Assume further that supports of $a^i_n, b^i_n,
c^i_n, d^i_n$, $i=1,2$ as functions of $n$ are either $0$ or all $\mathbb{Z}$.

 Then the ternary Virasoro-Witt algebras on $W$ defined by
the brackets \eqref{CFZbracket1}-\eqref{CFZbracket4}
   are ternary Hom-Nambu-Lie algebras  with twisting maps  $\alpha_1$ and $\alpha_2$, for any
value of $z$,  if and only if
\begin{eqnarray}
\alpha_1 (Q_n)= \beta_1   R_n,\quad \alpha_1 (R_n)= 0,\\
\alpha_2 (Q_n)= \beta_2   R_n,\quad \alpha_2(R_n)= 0.
\end{eqnarray}
\end{proposition}
\begin{remark}
We have considered previously only linear maps with a global support
but of course it could be possible to get other solutions if one
considers different supports for the maps.
\end{remark}

\section{Ternary Hom-Nambu-Lie algebras induced by Ternary Nambu-Lie algebras} \label{sec:TernHomNambuLieTernNambuLie}
In this Section, we address for ternary algebras
the following question discussed for binary
Hom-associative algebras by Y. Fr\'{e}gier and A. Gohr in
\cite{FregierGohr,Gohr}
(see also \cite{FregierGohrSilvestrov}): When a ternary
Hom-Nambu-Lie algebra $(V,[\cdot,\cdot,\cdot],(\alpha,\alpha))$ is induced by a
ternary Nambu-Lie algebra $(V,[\cdot,\cdot,\cdot]')$ by the composition method according to Theorem
\ref{ThmConstrHomNambuLie}? That is, when does exist a ternary algebra
endomorphism $\rho:V\rightarrow V$ with respect to $[\cdot,\cdot,\cdot]'$ such
that $[\cdot,\cdot,\cdot]=\rho\circ[\cdot,\cdot,\cdot]'$ and $\alpha=\rho$?

Let $(V,[\cdot,\cdot,\cdot],(\alpha,\alpha))$ be a ternary Hom-Nambu-Lie algebra
and
 assume that  $\alpha$ is a ternary
algebra endomorphism. The Hom-Nambu identity for any
$x_1,\cdots,x_5\in V$ is
\begin{gather}
HFI(x_{1},x_{2},x_{3},x_{4},x_{5})= \nonumber \\
[ \alpha (x_{1}),\alpha(x_{ 2}),[x_{3},x_{4},x_{5}]] - [ [x_1,x_2,x_3 ],\alpha
(x_{4}),\alpha(x_{5})]  \nonumber \\  - [ \alpha (x_3),[ x_1,x_2,x_{4}] ,
\alpha(x_{5})]- [\alpha (x_3),\alpha(x_{4}),[x_1,x_2,x_{5}]]=0.
\nonumber
\end{gather}
If $\alpha$ is invertible then it is equivalent to
\begin{equation}\label{HNidAlphaInvert}
\begin{array}{l}
[ x_{1},x_{ 2}, [\alpha ^{-1}(x_{3}),\alpha ^{-1}(x_{4}),\alpha^{-1}(x_{5})]] \\
- [ [\alpha^{-1}(x_1),\alpha ^{-1}(x_2),\alpha ^{-1}(x_3) ], x_{4},x_{5}]\\
- [ x_3,[ \alpha ^{-1}(x_1),\alpha ^{-1}(x_2),\alpha^{-1}(x_{4})] , x_{5}] \\
- [x_3,x_{4},[\alpha ^{-1}(x_1),\alpha^{-1}(x_2),\alpha ^{-1}(x_{5})]]=0.
\end{array}
\end{equation}
If $\alpha$ is a ternary algebra endomorphism, then $\alpha^{-1}$
is also a ternary algebra endomorphism. Indeed, for any
$x_1,x_2,x_3\in V$,
\begin{eqnarray*}\alpha
^{-1}([x_1,x_2,x_3])&=\alpha ^{-1}([\alpha\circ\alpha ^{-1}(x_1)
,\alpha\circ\alpha ^{-1}(x_2),\alpha\circ\alpha ^{-1}(x_3)])\\
\ &= \alpha ^{-1}\circ\alpha([\alpha ^{-1}(x_1) ,\alpha
^{-1}(x_2),\alpha ^{-1}(x_3)])\\
\ &= [\alpha ^{-1}(x_1) ,\alpha ^{-1}(x_2),\alpha ^{-1}(x_3)]
\end{eqnarray*}
Therefore, the identity \eqref{HNidAlphaInvert} can be written as
\begin{eqnarray}\label{HNidAlphaInvert2}
[ x_{1},x_{ 2}, \alpha ^{-1}([x_{3},x_{4},x_{5}])]- [ \alpha
^{-1}([x_1,x_2,x_3 ]), x_{4},x_{5}]\\
 - [ x_3,\alpha ^{-1}([ x_1,x_2,x_{4}]) , x_{5}]-
 [x_3,x_{4},\alpha ^{-1}([x_1,x_2,x_{5}])]=0. \nonumber
\end{eqnarray}
Assume now that the bracket $[\cdot,\cdot,\cdot]$ is induced by a ternary
Nambu-Lie algebra  $(V,[\cdot,\cdot,\cdot]')$ and a ternary algebra endomorphism
$\rho$, that is $[\cdot,\cdot,\cdot]=\rho\circ[\cdot,\cdot,\cdot]'$. It follows that the
identity \eqref{HNidAlphaInvert2} is
\begin{eqnarray}\label{HNidAlphaInvert3}
\rho([ x_{1},x_{ 2}, \alpha ^{-1}\circ\rho([x_{3},x_{4},x_{5}]')]'-
[ \alpha
^{-1}\circ\rho([x_1,x_2,x_3 ]'), x_{4},x_{5}]'\\
 - [ x_3,\alpha ^{-1}\circ\rho([ x_1,x_2,x_{4}]') , x_{5}]'-
 [x_3,x_{4},\alpha ^{-1}\circ\rho([x_1,x_2,x_{5}]')]')=0. \nonumber
\end{eqnarray}
The previous identity correspond to a Nambu identity if $\alpha
^{-1}\circ\rho$ is the identity map. Hence, we get the following
proposition.
\begin{proposition}
Let $(V,[\cdot,\cdot,\cdot],(\alpha,\alpha))$ be a ternary Hom-Nambu-Lie algebra
and assume that  $\alpha$ is a ternary
algebra automorphism.
Then the ternary Hom-Nambu-Lie algebra $(V,[\cdot,\cdot,\cdot],(\alpha,\alpha))$
is induced by a ternary Nambu-Lie algebra $(V,[\cdot,\cdot,\cdot]')$ where
$[\cdot,\cdot,\cdot]'=\alpha^{-1}[\cdot,\cdot,\cdot]$.
\end{proposition}

\begin{remark}
We observe that ternary Hom-Nambu-Lie $q$-Virasoro-Witt algebras
constructed in Theorem \ref{ThnHomNambuLieVW} could be untwisted
by using the linear map
\begin{equation}
\nu(Q_n)= q^{-n} Q_n \quad \nu(R_n)= q^{-n} R_n,
\end{equation}
since $\alpha_1=\alpha_2$ and $\alpha_1^{-1}=\alpha_2^{-1}=\nu$.
The ternary Hom-Nambu-Lie $q$-Virasoro-Witt algebras constructed
in Theorem \ref{thmHNL2} could not be untwisted in such a way (for $\beta_1=\beta_2$), as the linear maps
are nilpotent and thus are non-invertible. It would be of interest to know whether for $\beta_1=\beta_2$ it is possible to untwist these algebras in some other subtle ways.
\end{remark}

\end{document}